\documentclass[12pt]{article}
\usepackage{amsmath,amsfonts,amssymb,graphicx}
\usepackage{amsthm}
\usepackage{verbatim}

\title{Intermediate dimensions}
\author{Kenneth J. Falconer$^a$, Jonathan M. Fraser$^a$, and Tom Kempton$^b$}

\date{}

\def\rn{\mathbb{R}^n}

\renewcommand{\epsilon}{\varepsilon}

\newcommand\ubd{\overline{\mbox{\rm dim}}_{\rm B}\,} 
\newcommand\lbd{\underline{\mbox{\rm dim}}_{\rm B}\,} 
\newcommand\da{\underline{\mbox{\rm dim}}_{\,\theta}} 
\newcommand\db{\underline{\mbox{\rm dim}}_{\,\phi}} 
\newcommand\udb{\overline{\mbox{\rm dim}}_{\,\phi}} 
\newcommand\uda{\overline{\mbox{\rm dim}}_{\,\theta}} 
\newcommand\bdd{\mbox{\rm dim}_{\rm B}}
\newcommand\hdd{\mbox{\rm dim}_{\rm H}\,} 
\newcommand{\dimh}{\dim_{\rm H}}

\newcommand{\be}{\begin{equation}} 
\newcommand{\ee}{\end{equation}} 

 \newtheorem{theo}{Theorem}[section]
 
 \newtheorem{lem}[theo]{Lemma}
 \newtheorem{prop}[theo]{Proposition}

 \newtheorem{defi}[theo]{Definition}

\begin{document}
\maketitle

\vspace{-8mm}
\begin{center}
	$^a$Mathematical Institute, University of St Andrews, UK.\\
	 \vspace{3mm}
	$^b$School of Mathematics, University of Manchester, UK.\\
	 \vspace{3mm}
	 	\end{center}

\begin{abstract}
We introduce a continuum of dimensions which are `intermediate' between the familiar Hausdorff and box  dimensions.  This is done by restricting the families of allowable covers in the definition of Hausdorff dimension by insisting that $|U| \leq |V|^\theta$ for all sets $U, V$ used in a particular cover, where $\theta \in [0,1]$ is a parameter.  Thus, when $\theta=1$ only covers using sets of the same size are allowable, and we recover the box  dimensions, and when $\theta=0$ there are no restrictions, and we recover Hausdorff dimension.

We investigate many properties of the intermediate dimension (as a function of $\theta$), including proving that it is continuous on $(0,1]$ but not necessarily continuous at $0$, as well as establishing appropriate analogues of the mass distribution principle, Frostman's lemma, and the dimension formulae for products.  We also compute, or estimate, the intermediate dimensions of some familiar sets, including sequences formed by negative powers of integers, and Bedford-McMullen carpets.\\

\noindent \emph{Mathematics Subject Classification 2010}: primary: 28A80; secondary: 37C45.

\noindent \emph{Key words and phrases}: Hausdorff dimension, box  dimension, self-affine carpet.
 \end{abstract}

\section{Intermediate dimensions: definitions and background}
\setcounter{equation}{0}
\setcounter{theo}{0}

We work with subsets of $\rn$ throughout, although much of what we establish also holds in more general metric spaces. We denote the {\em diameter} of a set $F$ by $|F|$, and when we refer to a {\em cover} $\{U_i\}$ of a set $F$ we mean that $F\subseteq \bigcup_i U_i$ where $\{U_i\}$ is a finite or countable collection of sets.

Recall that  Hausdorff dimension $\dimh$ may be defined without introducing Hausdorff measures, but using Hausdorff content.  For  $F\subseteq \rn$,
\begin{align*}\label{hdim}
\dimh F = \inf \big\{& s\geq 0  :  \mbox{ \rm for all $\epsilon >0$ there exists a cover $ \{U_i\} $ of $F$ such that $\sum |U_i|^s \leq \epsilon$}  \big\},
\end{align*}
see \cite[Section 3.2]{Fa}. (Lower) box  dimension $\lbd$ may be expressed in a similar manner, by forcing the covering sets to be of the same diameter.  For  bounded $F\subseteq \rn$,
\begin{align*}
\lbd F =  \inf \big\{ s\geq 0  &:  \mbox{ \rm for all $\epsilon >0$ there exists a cover $ \{U_i\} $ of $F$} \\
 & \mbox{ \rm  such that  $|U_i| = |U_j|$ for all $i,j$  and  $\sum |U_i|^s \leq \epsilon$}  \big\}.
\end{align*}
see \cite[Chapter 2]{Fa}.  Expressed in this way, Hausdorff and box dimensions may be regarded as extreme cases of the same definition, one with no restriction on the size of covering sets, and the other requiring them all to have equal diameters.  With this in mind, one might regard them as the extremes of a continuum of dimensions with increasing restrictions on the relative sizes of covering sets.  This is the main idea of this paper, which we formalise by  considering restricted coverings where the diameters of the smallest and largest covering sets lie in a geometric range $\delta^{1/\theta} \leq |U_i| \leq \delta$ for some $0\leq \theta \leq 1$.
\begin{defi}\label{adef}
Let $F\subseteq \rn$ be bounded. For $0\leq \theta \leq 1$ we define the {\em  lower $\theta$-intermediate dimension} of $F$ by
\begin{align*}
\da F =  \inf \big\{& s\geq 0  :  \mbox{ \rm for all $\epsilon >0$ and all $\delta_0>0$,  there exists $0<\delta\leq \delta_0$} \\
 & \mbox{ \rm and a cover $ \{U_i\} $ of $F$ such that  $\delta^{1/\theta} \leq  |U_i| \leq \delta$ and 
 $\sum |U_i|^s \leq \epsilon$}  \big\}.
\end{align*}
Similarly, we define the {\em  upper $\theta$-intermediate dimension} of $F$ by
\begin{align*}
\uda F =  \inf \big\{& s\geq 0  :  \mbox{ \rm for all $\epsilon >0$ there exists $\delta_0>0$ such that for all $0<\delta\leq \delta_0$,} \\
 & \mbox{ \rm there is a cover $ \{U_i\} $ of $F$ such that  $\delta^{1/\theta} \leq  |U_i| \leq \delta$ and 
 $\sum |U_i|^s \leq \epsilon$}  \big\}.
\end{align*}
\end{defi}

With these definitions, 
$$\underline{\mbox{\rm dim}}_{0} F = \overline{\mbox{\rm dim}}_{0} F  = \hdd F, \quad \underline{\mbox{\rm dim}}_{1} F = \lbd F \quad {\mbox{ and }}\quad \overline{\mbox{\rm dim}}_{1} F = \ubd F, $$
where $\ubd$ is the upper box  dimension.  Moreover, it follows immediately that, for a bounded set $F$ and $\theta \in [0,1]$, 
\[
\hdd F \leq \da F \leq \uda F \leq \ubd F  \quad {\mbox{ and }}\quad \da F \leq  \lbd F.
\]
It is also immediate that $\da F$ and $\uda F$ are increasing in $\theta$, though as we shall see they need not be strictly increasing. Furthermore, $\uda$ is finitely stable, that is $\uda (F_1\cup F_2) = \max\{\uda F_1, \uda F_2\}$, and, for $\theta\in (0,1]$, both $\da F$ and $\uda F$ are unchanged on replacing $F$ by its closure. 

In many situations, even if $\hdd F < \ubd F$, we still have $\lbd F = \ubd F$ and $\da F = \uda F$ for all $\theta \in [0,1]$.  In this case we refer to the box  dimension $\bdd F = \lbd F = \ubd F$  and the {\em $\theta$-intermediate dimension} $\mbox{\rm dim}_{\theta} F = \da F = \uda F$.

This paper is devoted to understanding $\theta$-intermediate dimensions.  The hope is that $\mbox{\rm dim}_{\theta} F$ will interpolate between the Hausdorff and box  dimensions in a meaningful way, a rich and robust theory will be discovered, and interesting further questions unearthed. We first derive useful properties of intermediate dimensions, including that $\underline{\mbox{\rm dim}}_{0} F$ and  $\overline{\mbox{\rm dim}}_{0} F $ are continuous on $(0,1]$ but not necessarily at $0$, as well as proving versions of the mass distribution principle, Frostman's lemma and product formulae. We then examine a range of examples illustrating different types of behaviour including sequences formed by negative powers of integers, and self-affine Bedford-McMullen carpets.

 Intermediate dimensions provide an insight into the distribution of the diameters of covering sets needed when estimating the Hausdorff dimensions of sets whose Hausdorff and box dimensions differ.  They also have  concrete applications to well-studied problems.  For example, since the intermediate dimensions are preserved under bi-Lipschitz mappings,  they provide another invariant for Lipschitz classification of sets. A very specific variant was used in \cite{KP} to estimate the singular sets of partial differential equations.

A related approach to `dimension interpolation' was recently considered in \cite{spec1} where a new dimension function was introduced to interpolate between the box  dimension and the Assouad dimension.  In this case the dimension function was called the \emph{Assouad spectrum}, denoted by $\mbox{\rm dim}_\textup{A}^{\theta} F$ $(\theta \in (0,1))$.

\section{Properties of  intermediate dimensions}\label{sec:props}
\setcounter{equation}{0}
\setcounter{theo}{0}

\subsection{Continuity}

The first natural question is whether, for a fixed bounded set $F$, $\da F$ and $\uda F$  vary continuously for $\theta\in [0,1]$. We show this is the case, except possibly at $\theta = 0$.  We provide simple examples exhibiting discontinuity at $\theta = 0$ in Section \ref{examplessimple}.  However,  for many natural sets $F$ we find that the intermediate dimensions \emph{are} continuous at $0$ (and thus on $[0,1]$), for example for self-affine carpets, see Section \ref{carpetssec}.

\begin{prop}\label{cty}
Let $F$ be a non-empty bounded subset of $\rn$ and let $0\leq \theta<\phi  \leq 1$. Then 
\begin{equation}\label{ineqs}
\da F \ \leq \ \db F \ \leq \  \da F +\Big(1-\frac{\theta}{\phi}\Big) (n- \da F).
\end{equation}
and
\begin{equation}\label{ineqs2}
\uda F \ \leq \ \udb F \ \leq \  \uda F +\Big(1-\frac{\theta}{\phi}\Big) (n- \uda F).
\end{equation}
In particular,  $\theta \mapsto\da F$ and $\theta \mapsto\uda F$ are continuous for $\theta \in (0,1]$.
\end{prop}
\begin{proof}
We will only prove \eqref{ineqs} since \eqref{ineqs2} is similar. The left-hand inequality of \eqref{cty} is just the monotonicity of $\da F$. The right-hand inequality is trivially satisfied when $\da F=n$, so we assume that $0\leq \da F <n$. Suppose that $0\leq \theta<\phi  \leq 1$ and that  $0\leq \da F <s <n$. Then, given $\epsilon >0$, we may find arbitrarily small $\delta>0$ and countable or finite covers $\{U_i\}_{i\in I}$ of $F$ such that
\begin{equation}\label{epsum}
\sum_{i\in I} |U_i|^s<\epsilon \quad \mbox{ and } \quad \delta \leq |U_i| \leq \delta^\theta  \quad \mbox{ for all } i\in I.
\end{equation} 
Let
$$I_0 =\{i\in I: \delta \leq |U_i| \leq \delta^\phi\} \quad \mbox{ and } \quad I_1 =\{i\in I: \delta^\phi < |U_i| \leq \delta^\theta\}.$$ 
For each $ i\in I_1$ we may split $U_i$ into subsets of small coordinate cubes to get sets $\{U_{i,j}\}_{j\in J_i}$ such that $U_i \subseteq\bigcup_{j\in J_i} U_{i,j}$, with 
$|U_{i,j}| \leq \delta^\phi$ and ${\rm card} J_i \ \leq  \ c_n|U_i|^n \delta^{-\phi n} \  \leq  \ c_n\delta^{n(\theta-\phi)}$, where
$c_n = 4^n n^{n/2}$. 

Let $s<t\leq n$. Then $ \{U_i\}_{i\in I_0} \cup \{U_{i,j}\}_{i\in I_1, j \in J_i}$ is a cover of $F$ such that 
$\delta \leq |U_i|, |U_{i,j}| \leq \delta^\phi$. Taking sums with respect to this cover:
\begin{eqnarray*}
\sum_{i\in I_0}|U_i|^t  +  \sum_{i\in I_1}\sum_{j\in J_{i}}|U_{i,j}|^t 
&\leq& 
\sum_{i\in I_0}|U_i|^t  +  \sum_{i\in I_1} \delta^{\phi t} c_n |U_i|^n \delta^{-\phi n} \\
&\leq& 
\sum_{i\in I_0}|U_i|^t  + c_n \sum_{i\in I_1}  |U_i|^s  |U_i|^{n-s}\delta^{\phi( t-n)} \\
&\leq& 
\sum_{i\in I_0}|U_i|^s  + c_n \sum_{i\in I_1}  |U_i|^s  \delta^{\theta(n-s)}\delta^{\phi( t-n)} \\
&\leq& 
\sum_{i\in I_0}|U_i|^s  + c_n  \delta^{\phi[t-(n\phi+\theta(s-n))/\phi]}\sum_{i\in I_1}  |U_i|^s  \\
&\leq& (1+c_n) \sum_{i\in I}|U_i|^s\  < \ (1+c_n)\epsilon
\end{eqnarray*}
if $t\geq (n\phi+\theta(s-n))/\phi$, from \eqref{epsum}. This holds for some cover for arbitrarily small $\epsilon$ and  all $s>\da F$, giving 
$\db F  \leq n + \theta (\da F-n)/\phi$, which rearranges to give \eqref{ineqs}. 

Finally, note that \eqref{ineqs2} follows by exactly the same argument noting that the assumption $ \uda F <s $ gives rise to $\delta_0>0$ such that for all $\delta \in (0,\delta_0)$ we can find covers $\{U_i\}_{i\in I}$ of $F$ satisfying \eqref{epsum}.
\end{proof}

\subsection{A mass distribution principle for $\da$ and $\uda$}

The  \emph{mass distribution principle}  is a powerful tool in fractal geometry and provides a useful mechanism for estimating the Hausdorff dimension from below by considering measures supported on the set, see \cite[page 67]{Fa}.  We present  natural analogues for $\da$ and $\uda$.

\begin{prop}\label{mdp}
Let $F$ be a Borel subset of $\rn$ and let  $0\leq \theta \leq 1$ and $s\geq 0$. Suppose that there are numbers $a, c, \delta_0 >0$ such that for all $0< \delta\leq \delta_0$  we can find a Borel measure $\mu_\delta$ supported by $F$ with
$\mu_\delta (F) \geq a $, and with
\begin{equation}\label{mdiscond}
\mu_\delta (U) \leq c|U|^s \quad \mbox{ for all Borel sets  $U \subseteq \rn $ with } \delta \leq |U|\leq \delta^\theta.
\end{equation}
Then  $\da F \geq s$.  Moreover, if measures $\mu_\delta$ with the above properties can be found only for a sequence of $\delta \to 0$, then the conclusion is weakened to  $\uda F \geq s$. \end{prop}

\begin{proof}
Let $\{U_i\}$ be a cover of $F$ such that $\delta \leq |U_i|\leq \delta^\theta$ for all $i$. Then
$$a\ \leq \ \mu_\delta(F) \ \leq\  \mu_\delta\Big(\bigcup_i U_i\Big)\ \leq \ \sum_i \mu_\delta(U_i) \
\leq \  c\sum_i |U_i|^s,$$
so that $\sum_i |U_i|^s \geq a/c>0$ for every admissible cover and therefore $\da F \geq s$.

The weaker conclusion regarding the upper intermediate dimension is obtained similarly.
\end{proof}

Note  the main difference between Proposition \ref{mdp} and the usual mass distribution principle is that a family of measures $\{\mu_\delta\}$ is used instead of a single measure.  Since each measure $\mu_\delta$ is only required to describe a range of scales, in practice one can often use finite sums of point masses.  Whilst the measures $\mu_\delta$ may vary, it is essential that they all assign mass at least $a>0$ to $F$.

\subsection{A Frostman type lemma for $\da$}

\emph{Frostman's lemma}  is another  powerful tool in fractal geometry, which asserts the existence of measures of the type considered by the mass distribution principle, see \cite[page 77]{Fa} or \cite[page 112]{mattila}.  The following analogue of Frostman's lemma holds for intermediate dimensions and is a useful dual to Proposition \ref{mdp}.

\begin{prop}\label{frostman}
Let $F$ be a compact subset of $\rn$, let $0< \theta \leq 1$, and suppose $0< s< \da F$.  There exists a constant $c >0$ such that for all $\delta \in (0,1)$  we can find a Borel probability measure  $\mu_\delta$ supported on $F$ such that for all $x \in \rn$ and $\delta^{1/\theta} \leq r \leq \delta$, 
\[
\mu_\delta (B(x,r)) \leq c r^s .
\]
Moreover, $\mu_\delta$ can be taken to be a finite collection of atoms.
\end{prop}

\begin{proof}
This proof follows the proof of the classical version of Frostman's lemma given in \cite[pages 112-114]{mattila}.  For $m \geq 0$ let $\mathcal{D}_m$ denote the familiar partition of $[0,1]^n$ consisting of $2^{nm}$ pairwise disjoint half-open dyadic cubes of sidelength $2^{-m}$, that is cubes of the form $[a_1, a_1+2^{-m}) \times \cdots \times [a_n, a_n+2^{-m})$.  By translating and rescaling we may assume without loss of generality that $F \subseteq [0,1]^n$ and that $F$ is not contained in any  $Q \in \mathcal{D}_1$. It follows from the definition of $\da F$ that there exists $\varepsilon>0$ such that for all $\delta \in (0,1)$ and for all covers $\{U_i\}_i$ of $F$ satisfying $\delta^{1/\theta} \leq |U_i| \leq \delta$,
\begin{equation} \label{goodcover1}
\sum_i |U_i|^s > \varepsilon.
\end{equation}
Given $\delta \in (0,1)$, let $m \geq 0$ be the unique integer satisfying $2^{-m-1}<  \delta^{1/\theta} \leq 2^{-m}$ and let $\mu_m$ be a measure defined on $F$ as follows: for each $Q \in \mathcal{D}_m$  such that $Q \cap F \neq \emptyset$, then choose an arbitrary  point $ x_Q \in Q \cap F$ and let
\[
\mu_m = \sum_{ Q \in \mathcal{D}_m : Q \cap F \neq \emptyset} 2^{-ms} \delta_{x_Q}
\]
where $\delta_{x_Q}$ is a point mass at $x_Q$.  Modify $\mu_m$ to form a measure $\mu_{m-1}$, supported on the same finite set, defined by 
\[
\mu_{m-1}\vert_Q = \min\{1, 2^{-(m-1)s} \mu_m(Q)^{-1}\}  \mu_m\vert_Q
\]
for all $Q \in \mathcal{D}_{m-1}$,  where $\nu \vert_E$ denotes  the restriction of $\nu$ to $E$. The purpose of this modification is to reduce the mass of cubes which carry too much measure.  This is done since we are ultimately trying to construct a measure which we can estimate uniformly from above.   Continuing inductively, $\mu_{m-k-1}$ is obtained from $\mu_{m-k}$ by
\[
\mu_{m-k-1}\vert_Q = \min\{1, 2^{-(m-k-1)s} \mu_{m-k} (Q)^{-1}\} \mu_{m-k} \vert_Q
\]
for all $Q \in \mathcal{D}_{m-k-1}$.  We terminate this process when we define $\mu_{m-l}$ where $l$ is the largest integer satisfying $2^{-(m-l)}n^{1/2} \leq \delta$.  (We may assume that $l \geq 0$ by choosing $\delta$ sufficiently small to begin with.)  In particular, cubes $Q \in \mathcal{D}_{m-l}$ satisfy $|Q| = 2^{-(m-l)}n^{1/2} \leq \delta$. By construction we have
\begin{equation} \label{goodbound1}
\mu_{m-l}(Q) \leq 2^{-(m-k)s} = |Q|^s n^{-s/2}
\end{equation}
for all $k =0, \dots, l$ and $Q \in \mathcal{D}_{m-k}$.    Moreover, for all $x \in F$, there is at least one $k \in \{0, \dots, l\}$ and $Q \in \mathcal{D}_{m-k}$ with $x \in Q$ such that the inequality in \eqref{goodbound1} is an equality.  This is because all cubes at level $m$ satisfy the equality for $\mu_m$ and if a cube $Q$ satisfies the equality for $\mu_{m-k}$, then either $Q$ or its parent cube satisfies the equality for $\mu_{m-k-1}$. For each $x \in F$, choosing the largest such $Q$ yields a finite collection of cubes $Q_1, \dots, Q_t$ which cover $F$ and satisfy $\delta^{1/\theta} \leq |Q_i| \leq \delta$ for $i = 1, \dots, t$.  Therefore, using \eqref{goodcover1},
\[
\mu_{m-l}(F) = \sum_{i=1}^t \mu_{m-l}(Q_i) = \sum_{i=1}^t |Q_i|^s n^{-s/2} > \varepsilon  n^{-s/2}.
\]
Let $\mu_\delta = \mu_{m-l}(F)^{-1} \mu_{m-l}$, which is clearly a probability measure supported on a finite collection of points. Moreover, for all $x \in \rn$ and $\delta^{1/\theta} \leq r \leq \delta$, $B(x,r)$ is certainly contained in at most $c_n$ cubes in $\mathcal{D}_{m-k}$ where $k$ is chosen to be the largest integer satisfying $0 \leq k \leq l$ and  $2^{-(m-k+1)} < r$, and $c_n$ is a constant depending only on $n$.  Therefore, using \eqref{goodbound1},
\[
\mu_\delta (B(x,r)) \leq  c_n  \mu_{m-l}(F)^{-1}  2^{-(m-k) s} \leq    c_n \varepsilon^{-1}  n^{s/2} 2^s r^s
\]
which completes the proof, setting $c=  c_n \varepsilon^{-1}  n^{s/2} 2^s $.
\end{proof}

\subsection{General bounds}

Here we consider general bounds which rely on the Assouad dimension and which have interesting consequences for continuity.  Namely, they provide numerous examples where the intermediate dimensions are \emph{dis}continuous at $\theta=0$ and also provide another proof that the intermediate dimensions are continuous at $\theta = 1$. The \emph{Assouad dimension} of $F \subseteq \rn$ is defined by
\begin{eqnarray*}
\dim_\textup{A} F &=& \inf \bigg\{ s \geq 0  \ : \ \text{there exists $C>0$ such that for all $x \in F$, } \\
&\,& \qquad \qquad   \text{and for all $0<r<R$, we have} \  N_r(F \cap B(x,R)) \leq C \left(\frac{R}{r} \right)^s \bigg\}
\end{eqnarray*}
where $N_r(A)$ denotes the smallest number of sets of diameter at most $r$ required to cover a set $A$.  In general  $\underline{\dim}_\textup{B} F \leq \overline{\dim}_\textup{B} F \leq \dim_\textup{A} F \leq n$, but equality of these three dimensions occurs in many cases, even if the Hausdorff dimension and box dimension are distinct, for example if the box dimension is equal to  the ambient spatial dimension.   See \cite{Fra, robinson} for more background on the Assouad dimension. 
The following proposition gives lower bounds for the intermediate dimensions in terms of Assouad dimensions.

\begin{prop} \label{assouad}
Given any non-empty bounded $F \subseteq \rn$ and $\theta \in (0,1)$, 
\[
\da F \geq \dim_\textup{A} F - \frac{\dim_\textup{A} F - \underline{\dim}_\textup{B} F}{\theta},
\]
and
\[
\uda F \geq \dim_\textup{A} F - \frac{\dim_\textup{A} F - \overline{\dim}_\textup{B} F}{\theta}.
\]
In particular, if $\underline{\dim}_\textup{B} F = \dim_\textup{A} F$, then $\da F =\overline{\dim}_\theta F = \underline{\dim}_\textup{B} F = \dim_\textup{A} F$ for all $\theta \in (0,1]$. 
\end{prop}

\begin{proof}
We will prove the lower bound for $\da F$, the proof for $\uda F$  is similar.  Fix $\theta \in (0,1)$ and assume that $\underline{\dim}_\textup{B} F >0$, otherwise the result is trivial.  Let
\[
0<b< \underline{\dim}_\textup{B} F  \leq  \dim_\textup{A} F < d < \infty
\]
and $\delta \in (0,1)$ be given.  By the definition of lower box dimension, there exists a uniform constant $C_0$, depending only on $F$ and $b$, such that there is a $\delta$-separated set of points in $F$ of cardinality at least $C_0 \delta^{-b}$.  Let $\mu_\delta$ be a uniformly distributed probability measure on these points, i.e. a sum of $C_0\delta^{-b}$ point masses each with mass $C_0^{-1} \delta^{b}$. We use our mass distribution principle with this measure to prove the proposition.

Let $U \subseteq \rn$ be a Borel set with $|U| = \delta^\gamma$ for some $\gamma \in [\theta, 1]$.  By the definition of Assouad dimension there exists a uniform constant $C_1$, depending only on $F$ and $d$, such that $U$ intersects at most $C_1 ( \delta^\gamma/\delta)^d$ points in the support of $\mu_\delta$.  Therefore
\[
\mu_\delta (U) \leq C_1\delta^{(\gamma-1)d} C_0^{-1} \delta^{b} = C_1C_0^{-1} \ |U|^{(\gamma d - d+ b)/ \gamma} \leq  C_1C_0^{-1}  \ |U|^{(\theta d - d+ b)/ \theta} 
\]
which, using Proposition \ref{mdp}, implies that
\[
\da F \geq (\theta d - d+ b)/ \theta = d - \frac{d-b}{\theta}.
\]
Letting $d \to  \dim_\textup{A} F$ and $b \to \underline{\dim}_\textup{B} F$ yields the desired result.
\end{proof}

This proposition implies that for bounded sets with $\dim_\textup{H} F < \underline{\dim}_\textup{B} F = \dim_\textup{A} F$, the intermediate dimensions $\da F$ and $\uda F$ are necessarily discontinuous at $\theta = 0$.  In fact the intermediate dimensions are constant on $(0,1]$ in this case.  On the other hand, this gives an alternative demonstration that $\da F$ and $\uda F$  are \emph{always} continuous at $\theta=1$.  Moreover, the proposition provides a quantitative lower bound near $\theta=1$.  In Section \ref{examplessimple} we will use Proposition \ref{assouad} to construct examples exhibiting a range of behaviours.

\subsection{Product formulae}

A well-studied problem in dimension theory is how dimensions of product sets behave.  The following product formulae for  intermediate dimensions may be of interest in their own right, but in Section \ref{examplessimple} they will be used to construct  examples.

\begin{prop} \label{products}
Let $E \subseteq \rn$ and $F \subseteq \mathbb{R}^m$ be bounded and $\theta \in [0,1]$.  Then
\[
\da E + \da F\ \leq\ \da (E \times F)\    \leq\ \overline{\dim}_\theta ( E \times F)\ \leq\ \overline{\dim}_\theta  E + \overline{\dim}_\textup{B} F.
\]
\end{prop}

\begin{proof}
Fix $\theta \in (0,1)$ throughout, noting that the cases when $\theta=0,1$ are well-known, see \cite[Chapter 7]{Fa}.  We begin by demonstrating the left-hand inequality.  We may assume that   $\da E ,  \da F  >0$ as otherwise the conclusion follows by monotonicity.  Moreover, since $E,F$ are bounded we may assume they are compact since all the dimensions considered are unchanged under taking closure.  Let $0< s< \da E$ and $0<t< \da F$.  It follows from Proposition \ref{frostman} that there exist constants $C_s,C_t > 0$ such that for all $\delta \in (0,1)$  there exist Borel probability measures $\mu_\delta$ supported on $E$ and $\nu_\delta$ supported on $F$ such that for all $x \in \rn$ and $\delta^{1/\theta} \leq r \leq \delta$,
\[
\mu_\delta(B(x,r)) \leq C_s r^s \qquad \text{and} \qquad \nu_\delta(B(x,r)) \leq C_t r^t.
\]
Consider the product measure $\mu_\delta \times \nu_\delta$ which is supported on $E \times F$.  For $z \in \rn \times  \mathbb{R}^m$ and $\delta^{1/\theta} \leq r \leq \delta$,
\[
(\mu_\delta \times \nu_\delta) (B(z,r)) \leq C_s C_t  r^{s+t}
\]
and Proposition \ref{mdp} yields $\da (E \times F) \geq s+t$; letting  $  s\to \da E$ and $t\to \da F$ gives the  desired inequality.

The middle inequality is trivial and so it remains to prove the right-hand inequality.  Let $s > \overline{\dim}_\theta  E$ and $d > \overline{\dim}_\textup{B} F$.  From the definition of $\overline{\dim}_\textup{B} F$ there exists a constant $\delta_1 \in (0,1)$ such that for all $0<r<\delta_1$ there is a cover of $F$ by at most $r^{-d}$ sets of diameter $r$.
Let $\varepsilon>0$. By the definition of $\overline{\dim}_\theta  E$ there exists  $\delta_0\in (0,\delta_1)$ such that for all $0<\delta < \delta_0$  there is a cover of $E$ by sets $\{U_i\}_{i}$ with $\delta^{1/\theta} \leq |U_i| \leq \delta$ for all $i$ and 
\[
\sum_{i} |U_i|^s \leq \varepsilon.
\]
Given such a cover of $E$, for each $i$ let $\{U_{i,j}\}_{j}$ be a cover of $F$ by at most $|U_i|^{-d}$ sets with 
diameters $|U_{i,j}|=|U_{i}|$ for all $j$.  Then
$$ E\times F\ \subseteq\  \bigcup_i \bigcup_j \big(U_i \times U_{i,j}\big),$$ 
with
$$ \sum_i\sum_j |U_i \times U_{i,j}|^{s+d} \ \leq\ \sum_i|U_i|^{-d}\big( \sqrt{2}|U_i|\big)^{s+d}
  \ =\ 2^{(s+d)/2}\sum_i|U_i|^{s}  \ \leq\ 2^{(s+d)/2}\varepsilon.$$
Since
$\delta^{1/\theta} \leq |U_i \times U_{i,j}| \leq \sqrt{2}\delta$ for all $i,j$, each set $U_i \times U_{i,j}$ may be covered by at most $c_{n+m}$ sets $\{V_{i,j,k}\}_k$ with diameters  $\delta^{1/\theta}\leq |V_{i,j,k}| \leq \min\{|U_i \times U_{i,j}|,\delta\}\leq \delta$, where $c_{q}$ is the least number such that  every set  in $\mathbb{R}^q$ of diameter $\sqrt{2}$ can be covered by at most $c_{q}$ sets of diameter 1. Hence
$$\sum_i\sum_j \sum_k|V_{i,j,k}|^{s+d}\ \leq \ c_{n+m}2^{(s+d)/2}\varepsilon.$$
As $\varepsilon$ may be taken arbitrarily small, $\overline{\dim}_\theta  (E \times F) \leq s+d$; letting $  s\to \uda E$ and $d\to \overline{\dim}_\textup{B} F$ completes the proof.
\end{proof}

\section{Examples}\label{sec:egs}
\setcounter{equation}{0}
\setcounter{theo}{0}

In this section we construct several simple examples where the intermediate dimensions exhibit a range of phenomena.  All of our examples are compact subsets of $\mathbb{R}$ or $\mathbb{R}^2$ and in all examples the upper and the lower intermediate dimensions coincide for all $\theta \in [0,1]$. 

\subsection{Convergent sequences}\label{secseq}
Let $p>0$ and
\begin{equation*}
F_p = \bigg\{0, \frac{1}{1^p}, \frac{1}{2^p},\frac{1}{3^p},\ldots \bigg\}.
\end{equation*}
Since $F_p$ is countable, $\hdd F_p=0$. It is well-known that $\bdd F_p=1/(p+1)$, see \cite[Chapter 2]{Fa}. We  obtain the intermediate dimensions of $F_p$.
\begin{prop}\label{conseq}
For $p>0$ and $0\leq \theta \leq 1$,
$$\da F_p  = \overline{\dim}_\theta F_p = \frac{\theta}{p+\theta}.$$
\end{prop}

\begin{proof}
We first bound  $\overline{\dim}_\theta F_p$ above. Let $0<\delta<1$ and let $M =\lceil \delta^{-(s +\theta(1-s))/(p+1)}\rceil$. Write $B(x,r)$ for the closed interval (ball) of centre $x$ and length $2r$.
Take a covering ${\mathcal U}$ of  $F_p$ consisting of $M$ intervals $B(k^{-p}, \delta/2)$  of length $\delta$ for $1\leq k\leq M$ and $\lceil M^{-p} /\delta^\theta\rceil \leq M^{-p}/\delta^\theta +1$ intervals  of length $\delta^\theta$  that cover $[0, M^{-p}].$   Then
\begin{eqnarray*} 
\sum_{U\in {\mathcal U}} |U|^s &\leq & M\delta ^s + \delta^{\theta s}\Big(\frac{1}{M^p\delta^\theta}+ 1\Big)\\ 
 &=&  M\delta ^s + \frac{\delta^{\theta (s-1)}}{M^p}+ \delta^{\theta s}\\
 & \leq & ( \delta^{-(s +\theta(1-s))/(p+1)} +1) \delta^s 
+ \delta^{\theta (s-1)}\delta^{(s +\theta(1-s))p/(p+1)} +\delta^{\theta s}\\
& =&  2\delta^{(\theta (s-1)+sp)/(p+1)} +  \delta^s+ \delta^{\theta s} \ \to \ 0
\end{eqnarray*}
as $\delta \to 0$ if $s(\theta +p) > \theta$. Thus $\overline{\dim}_\theta F_p \leq\theta/(p+\theta)$.  
\medskip

For the lower bound we put suitable measures on $F_p$ and apply Proposition \ref{mdp}. Fix $s = \theta/(p+\theta)$. Let $0<\delta < 1$ and again let $M =\lceil \delta^{-(s +\theta(1-s))/(p+1)}\rceil$.  Define $\mu_\delta$ as the sum of point masses on the points $1/k^p \ (1\leq k<\infty)$ with
\begin{equation}\label{mass}
 \mu_\delta \Big(\left\{\frac{1}{k^p}\right\}\Big)\  = \ 
\left\{
\begin{array}{cl}
 \delta^s & \mbox{ if } 1\leq k \leq  M   \\
 0 &    \mbox{ if } M+1\leq k <\infty  
\end{array}
\right. .
\end{equation}
Then
\begin{eqnarray*} 
\mu_\delta(F_p) &=& M\delta^s \\
& \geq&\delta^{-(s +\theta(1-s))/(p+1)}\delta^s \\
& =& \delta^{(ps +\theta(s-1))/(p+1)}  = 1 
\end{eqnarray*}
by the choice of $s$. 

To see that \eqref{mdiscond} is satisfied, 
note that if $2\leq k\leq M$ then, by a mean value theorem estimate,
$$\frac{1}{(k-1)^p} - \frac{1}{k^p}\ \geq \ \frac{p}{k^{p+1}}\ \geq \ \frac{p}{M^{p+1}};$$
thus the gap between any two points of $F_p$ carrying mass is at least $p/M^{p+1}$.
Let $U$ be such that $\delta \leq |U|\leq \delta^\theta$. Then $U$ intersects at most 
$1+ |U|/(p/M^{p+1})  = 1+|U|M^{p+1}/p $ of the points of $F_p$ which have mass $\delta^s$.
Hence 
\begin{eqnarray*} 
\mu_\delta (U)& \leq & \delta^s + \frac{1}{p}|U|\delta^s \delta^{-(s +\theta(1-s))} \\ 
&=&   \delta^s + \frac{1}{p}|U| \delta^{(\theta(s-1))}\\
&\leq& |U|^s + \frac{1}{p} |U||U|^{s-1} \\ 
&=& \left(1+ \frac{1}{p}\right)|U|^s. 
\end{eqnarray*}
From Proposition  \ref{mdp}, $\da F_p \geq s = \theta/(p+\theta)$.
\end{proof}

\subsection{Simple examples exhibiting different phenomena} \label{examplessimple}

We use the convergent sequences $F_p$ from the previous section, the product formulae from Proposition \ref{products}, and that the \emph{upper} intermediate dimensions are finitely stable to construct examples displaying a range of different features which are illustrated in Figure 1.

The natural question which began this investigation is `does $\da$ vary continuously between the Hausdorff and lower box dimension?'.  This indeed happens for the convergent sequences considered in the previous section, but turns out to be false in general.  The first example in this direction is provided by another convergent sequence.
\medskip

\begin{figure}[t]
\centering
\includegraphics[width=0.65\textwidth]{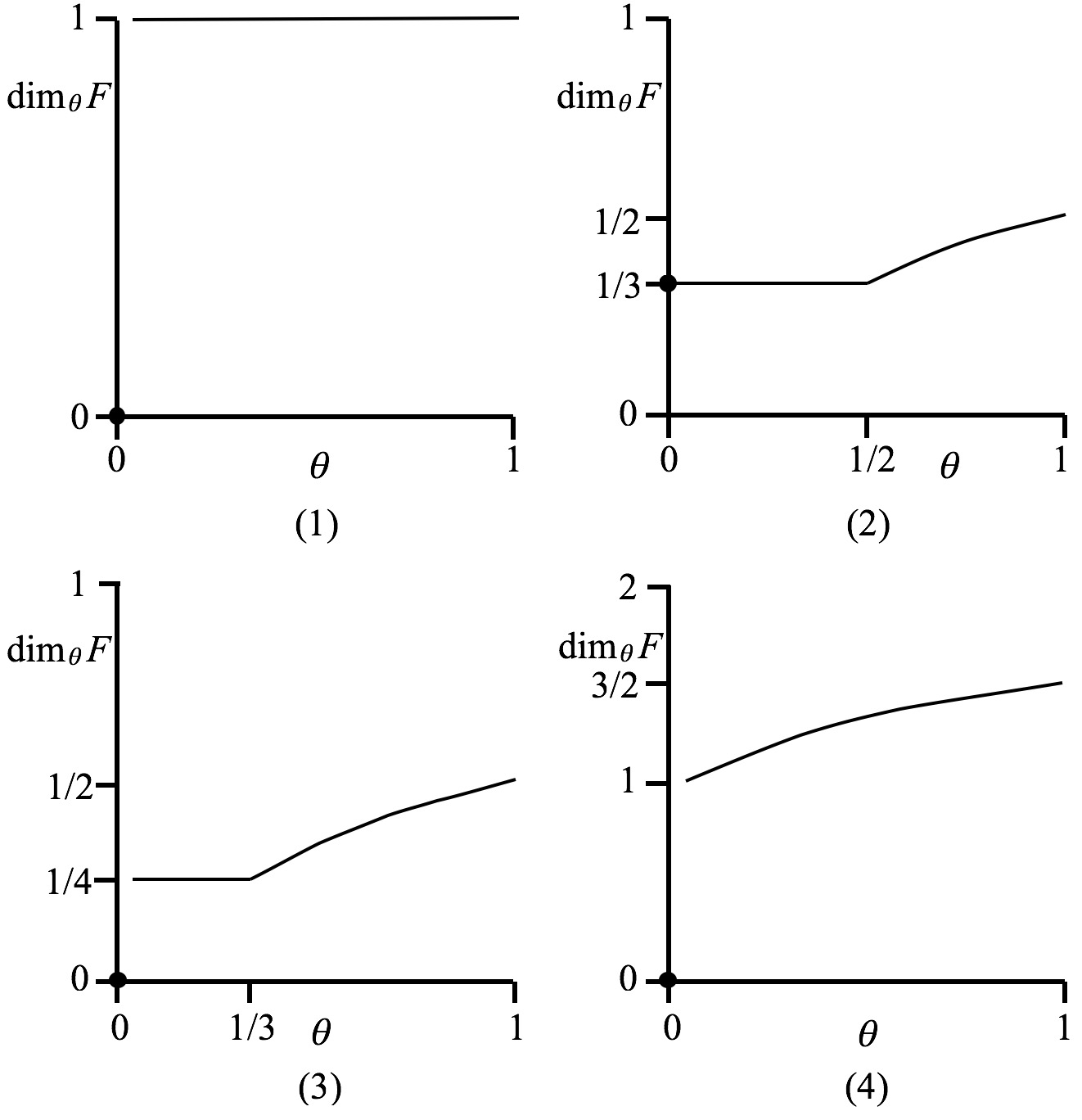}
\caption{Plots of  $\da F$ for the four examples in Section \ref{examplessimple}}
\end{figure}

\emph{Example 1: Discontinuous at 0, otherwise constant.}   Let
\[
F_{\log} = \{ 0, 1/\log 2, 1/\log 3, 1/\log 4, \dots \}.
\]
This sequence converges slower than any of the polynomial sequences $F_p$ and it is well-known and easy to prove that
$\bdd F_{\log} = \dim_\textup{A} F_{\log}= 1$.
It follows from Proposition \ref{assouad} that 
\[
\da F_{\log} = \overline{\dim}_\theta F_{\log} =  1,\qquad \theta \in (0,1].
\]
 Since $\dim_0 F_{\log}=\hdd F_{\log}=0$ there is a discontinuity at $\theta = 0$.
\medskip

\emph{Example 2: Continuous at 0, part constant, part strictly increasing. }  In the opposite direction, it is possible that  $\da F = \hdd F < \lbd F$ for some $\theta >0$.  Indeed, let $F =F_1 \cup E$ where $F_1 = \{0, 1/1, 1/2, 1/3, \dots\}$ as before, and let $E \subset \mathbb{R}$ be any compact set with $\hdd E = \ubd E = 1/3$ (for example an appropriately chosen self-similar set).  Then it is straightforward to deduce that 
\[
\da F = \overline{\dim}_\theta F =  \max\left\{ \frac{ \theta}{1+ \theta},  \, 1/3\right\},\qquad \theta \in [0,1].
\]
It is also possible for $\da F$ to approach a value strictly in between $\hdd F$ and $\lbd F$ as $\theta \to 0$.  This is the subject of the next two examples.
\medskip

\emph{Example 3:  Discontinuous at 0, part constant, part strictly increasing.}  For an example that is constant on an interval adjacent to 0, let $F =E \cup   F_{1}$ where this time $E\subset \mathbb{R}$ is any closed countable set with $\lbd E = \dim_\textup{A} E = 1/4$.  It is immediate  that
\[
\da F = \overline{\dim}_\theta F =  \max\left\{ \frac{ \theta}{1+ \theta},  \, 1/4\right\},\qquad \theta \in (0,1],
\]
with $\hdd F=0$ and $\bdd F=1/2$.

\medskip

\emph{Example 4: Discontinuous at 0, strictly increasing.} Finally, for an example where $\da F$ is smooth, strictly increasing but not continuous at $\theta=0$,  let
\[
F = F_{1} \times F_{\log}\subset \mathbb{R}^2.
\]
Here $\hdd F=0$ and $\bdd F=3/2$ and Proposition \ref{products} gives
\[
\da F = \overline{\dim}_\theta F =  \frac{ \theta}{1 + \theta}+ 1,\qquad \theta \in (0,1],
\]
noting that  $\da F_{\log} = \overline{\dim}_\theta F_{\log} = \lbd F_{\log} = \ubd F_{\log} = \dim_\textup{A} F_{\log}= 1$ for $\theta \in (0,1]$.

\section{Bedford-McMullen carpets} \label{carpetssec}

A well-known class of fractals where the Hausdorff and box dimensions differ are the self-affine carpets; this is a consequence of the alignment of the component rectangles in the iterated construction. The first studies of planar self-affine carpets were by Bedford \cite{bedford} and McMullen \cite{mcmullen} independently, see also \cite{peres} and these Bedford-McMullen carpets have been widely studied and generalised, see for example \cite{Fa1} and references therein.  Indeed, the question of the distribution of scales of covering sets for Hausdorff and box dimensions of Bedford-McMullen carpets was one of our motivations for studying intermediate dimensions. Finding an exact formula for the intermediate dimensions of Bedford-McMullen carpets seems a difficult problem, so here we obtain some lower and upper bounds, which in particular establish continuity at $\theta=0$ and that  the intermediate dimensions take a strict minimum at $\theta=0$. 

We introduce the notation we need, generally following that of McMullen \cite{mcmullen}. Choose integers $n > m \geq 2$.  Let ${I}=\left\{0,\ldots, m-1 \right\}$ and ${J}=\left\{0,\ldots, n-1 \right\}$.  Choose a fixed digit set $D \subseteq I \times J$ with at least two elements.  For $\left( p,q \right)\in D$ we define the affine contraction $S_{\left(p,q \right)}\colon [0,1]^2 \rightarrow [0,1]^2$ by
\[
S_{ \left( p,q \right)}\left(x,y\right) = \left( \frac{x+p}{m},\frac{y+q}{n} \right) .
\]
There exists a unique non-empty compact set $F \subseteq [0,1]^2$ satisfying 
\[
F=\bigcup_{(p,q) \in D}S_{ \left( p,q \right)}(F);
\]
that is $F$ is the attractor of the iterated function system $\ \left\{ S_{ \left(p,q \right)} \right\}_{ \left( p,q \right)\in D}$.  We call such a set $F$ a {\em Bedford-McMullen self-affine carpet}.   It is sometimes convenient to denote pairs in $D$ by $\ell = (p_\ell,q_\ell )$.  

We  model our carpet $F$ via the symbolic space $D^{\mathbb{N}}$, which consists of all infinite words over $D$ and is equipped with the product topology.  
We write $\mathbf{i}  \equiv (i_1,i_2,\ldots)$ for elements of  $D^{\mathbb{N}}$ and  $(i_1,\ldots,i_k)$ for words of length $k$ in $D^k$, where $i_j \in D$.
Then the canonical projection  $\tau :D^{\mathbb{N}} \rightarrow [0,1]^2$ is defined by
\[
\{\tau(\mathbf{i})\}\equiv\{\tau(i_1,i_2,\ldots)\}=\bigcap_{k \in \mathbb{N}} S_{i_1} \circ \cdots  \circ S_{i_k}([0,1]^2).
\]
 This allows us to switch between symbolic and geometric notation since
\[
\tau(D^{\mathbb{N}})=F.
\]

Bedford \cite{bedford} and McMullen \cite{mcmullen} showed that 
\be\label{dimb}
\bdd F = \frac{\log m_0}{\log m} + \frac{\log |D| - \log m_0}{\log n} 
\ee
where $m_0$ is the number of $p$ such that there is a $q$ with $(p,q) \in D$, that is the number of columns of the array containing at least one rectangle.
They also showed that 
\be\label{dimh}
\hdd F = \log \big(\sum_{p=1}^m n_p^{\log_n m} \big) \Big/ \log m,
\ee
where $n_p \, (1\leq p\leq m)$ is the number of $q$ such that $(p,q)\in D$, that is the number of selected rectangles  in the $p$th column of the array.

For each $\ell \in D$ we let $a_\ell$ be the number of rectangles of $D$ in the same column as $\ell$. 
Then, writing  $d = \hdd F$, \eqref{dimh} may be written as
\begin{equation}
m^d \ = \ \sum_{p=1}^m {n}_p^{\log_n m}\ =\  \sum_{\ell\in D} a_{\ell}^{(\log_n m -1)}, 
\end{equation}
where  equality of the sums follows from the definitions of $n_p$ and $a_\ell$.

We assume that the non-zero $n_p$ are not all equal, otherwise $\hdd F = \bdd F$;  in particular this implies that $a := \max_{\ell \in D}a_\ell \geq 2$.

We denote the $k$th-level iterated rectangles by
$$  R_k (i_1,\ldots,i_k)\  =\  S_{i_1}\circ \cdots \circ S_{i_k}([0,1]^2).$$
We also write $R_k(\mathbf{i})  \equiv R_k(i_1,i_2, \ldots)$ for this rectangle when we wish to indicate the $k$th-level iterated rectangle containing the point $\tau(\mathbf{i}) = \tau(i_1,i_2, \ldots)$.

 We will associate a probability vector $\ \left\{ b_{\ell} \right\}_{\ell\in D}$ with $D$ and let $\widetilde{\mu}=\prod_{k\in\mathbb{ N}}\left( \sum_{ i_k\in D}b_{  i_k}\delta_{  i_k}\right)$  be the natural Borel product probability measure on $D^{\mathbb{N}}$, where $\delta_{ \ell}$ is the Dirac measure on $D$ concentrated at $ \ell$.  Then the measure 
\[
\mu=\widetilde{\mu}\circ \tau^{-1}
\]
 is a self-affine measure supported on $F$. 
Following McMullen \cite{mcmullen} we set 
\be\label{pi}
b_\ell = a_\ell^{(\log_n m -1)}/m^d\qquad (\ell \in D),
\ee
noting that $\sum_{\ell\in D} b_\ell  =1$, to get  a measure $\mu$ on $F$; thus  the measures of the iterated rectangles are
\be\label{rect}
\mu\big( R_k (i_1,\ldots,i_k)\big)\ =\  b_{i_1}\cdots b_{i_k}\ 
=\  m^{-kd} ( a_{i_1} \cdots a_{i_k})^{(\log_n m -1)}.
\ee

Approximate squares are  well-known tools in the study  of self-affine carpets.  Given $k\in \mathbb{N}$ let $l(k) = \lfloor k \log_n m\rfloor $ so that 
\be\label{lkbounds}
k\log_n m \ \leq\ l(k) \ \leq k\log_n m \ + \ 1
\ee
For such $k$ and  $\mathbf{i}  = (i_1,i_2, \ldots) \in D^{\mathbb{N}}$ we define  the {\em approximate square} containing $\tau(\mathbf{i})$ as the union of $m^{-k}\times n^{-k}$ rectangles:
\begin{align*}
Q_k(\mathbf{i}) = Q_k(i_1,i_2, \ldots) = \bigcup\Big\{ &R_k(i_1',\ldots,i_k'): p_{i_{j}'}=p_{i_{j}}
\text{ for } j= 1, \ldots, k \\
 &  \text{ and } q_{i_{j}'}=q_{i_{j}}  \text{ for }   j= 1, \ldots, l(k)  \   \Big\},
\end{align*}
recalling that $\ell = (p_\ell,q_\ell)$.
 This approximate square has sides $m^{-k}\times n^{-l(k)}$ where 
$n^{-1}m^{-k} \leq n^{-l(k)}\leq m^{-k}$. 

Note that, by virtue of self-affinity and since the sequence $(p_{i_1},\cdots p_{i_k})$ is the same for all level-$k$ rectangles $R_k (i_1,\ldots,i_k)$ in the same approximate square, the $a_{i_{l(k)+1}}\cdots a_{i_k}$ level-$k$ rectangles that comprise the  approximate square $Q_k(\mathbf{i})$  all have equal $\mu$-measure. Thus, writing $L= \log_n\! m$,
\begin{eqnarray}
\mu(Q_k(\mathbf{i})) &=&   m^{-kd} a_{i_1}^{(L -1)}\cdots a_{i_k}^{(L -1)}
\times a_{i_{l(k)+1}}\cdots a_{i_k}\label{squaremes1}\\
&=&  m^{-kd} a_{i_1}^{L}\cdots a_{i_{k}}^{L}
\times a_{i_{1}}^{-1}\cdots a_{i_{l(k)}}^{-1}.\label{squaremes}
\end{eqnarray}

We now obtain an upper bound for  $\overline{\dim}_\theta F$ which implies continuity at $\theta = 0$ and so on $[0,1]$.  Recall that $a = \max_{\ell \in D}a_\ell \geq 2$.

\begin{prop}\label{propup}
Let $F$ be the Bedford-McMullen carpet as above. Then for $0<\theta <\frac{1}{4}(\log_n\! m)^2$,
\be\label{upbound}
 \overline{\dim}_\theta F \ \leq \  \hdd F +  \bigg(\frac{2\log(\log_m n)\log a}{\log n}\bigg) \frac{1}{-\log \theta}.
 \ee
In particular, $ \underline{\dim}_\theta F$ and $ \overline{\dim}_\theta F$ are continuous at $\theta = 0$ and so are continuous on $[0,1]$.
\end{prop}

\begin{proof}
For $\mathbf{i}  = (i_1,i_2,\ldots)$, rewriting  \eqref{squaremes} gives
\be
\mu\big(Q_k(\mathbf{i})\big) =   m^{-kd} 
\bigg(\frac{ (a_{i_1}\cdots a_{i_k})^{1/k}}{(a_{i_1}\cdots a_{i_{l(k)}})^{1/l(k)}}\bigg)^{Lk}
\big(a_{i_1}\cdots a_{i_{l(k)}}\big)^{(kL/l(k))-1}.
\label{rewritemes}
\ee
We consider the two bracketed terms on the right-hand side of \eqref{rewritemes} in turn. We show that  the first term cannot be too small for too many consecutive $k$ and that the second term is bounded below by $1$.

For $k>1$ with $l(k) =\lfloor k \log_n m\rfloor$ as usual, define
\be\label{fkstar}
f_k(\mathbf{i})\ \equiv\ f_k(i_1,i_2,\ldots)  
 \ =\ \bigg(\frac{ (a_{i_1}\cdots a_{i_k})^{1/k}}{(a_{i_1}\cdots a_{i_{l(k)}})^{1/l(k)}}\bigg)
\ee
We claim that for all $K\geq L/(1-L)$ and all $\mathbf{i} = (i_1,i_2,\ldots)  \in D^{\mathbb{N}}$,   there exists   $k$  with $K\leq k \leq K/ \theta$   such that  
\be\label{claim}
 f_k(\mathbf{i})\ \geq \ a^{\log L/\log(L/2\theta)}.
\ee
 Suppose this is false for some $(i_1,i_2,\ldots)$ and $K\geq 1/L(L-1)$, so for all $K\leq k \leq K/ \theta$, $  f_k(\mathbf{i}) < \lambda:= a^{\log L/\log(L/2\theta)}$, that is
\be\label{basin}
(a_{i_1} a_{i_2}\cdots a_{i_k})^{1/k}\ <\ \lambda (a_{i_1} a_{i_2}\cdots a_{i_l})^{1/l(k)}.
\ee
Define a sequence of integers $k_{r}\, (r=0,1,2,\ldots)$ inductively by $k_{0} = K$ and  for $r \geq 1$ taking $k_{r}$ to be the least integer such that 
$ \lfloor k_{r} \log_n m \rfloor = k_{r-1}$. 
Then  $k_{r} \leq  k_{r-1}/L  +1$, and a simple induction shows that
$$k_{r}\  \leq \  L^{-r}\big(K + L/(1-L)\big)-L/(1-L) \ \leq\  L^{-r}\big(K + L/(1-L)\big)\qquad (r\geq 0).$$
Fix $N$ to be the greatest integer such that $k_N \leq K/\theta$. Then
$$ L^{-(N+1)} \big(K+L\big/(1 -L)\big)\ \geq\ k_{N+1}\ >\ K/\theta$$
so rearranging, provided $K\geq L/(1-L)$, 
$$L^N\ <\  \frac{\theta\big(K + L\big/(1-L)\big)}{LK}\ \leq \ \frac{2\theta}{L}, 
$$
that is
\be\label{ineq1}
N > \frac{\log(2\theta/L)}{\log L}.
\ee
From \eqref{basin}
$$
(a_{i_1} a_{i_2}\cdots a_{i_{k_r}})^{1/{k_r}}\ <\ \lambda (a_{i_1} a_{i_2}\cdots a_{i_{k_{r-1}}})^{1/k_{r-1}} \qquad (1\leq r\leq N)
$$
so iterating
$$
 1\ \leq\ (a_{i_1} a_{i_2}\cdots a_{i_{k_N}})^{1/k_N}\ <\ \lambda^{N} (a_{i_1} a_{i_2}\cdots a_{i_{K}})^{1/K}\ \leq \ \lambda^{N}a.
$$
Combining with \eqref{ineq1} we obtain
$$
\lambda \ >\ a^{-1/N} \geq a^{\log L/\log(L/2\theta)}
$$
which contradicts the definition of $\lambda$. Thus the claim \eqref{claim} is established.

For the second bracket on the right-hand side of \eqref{rewritemes} note that
$$ 0\ \leq\ (kL/l(k))-1\ =\ \frac{ k\log_n m -\lfloor k\log_n m \rfloor}{l(k)}\ \leq \frac{1}{l(k)}$$
so that 
\be\label{brac2}
1\ \leq \  \big(a_{i_1}\cdots a_{i_{l(k)}}\big)^{(kL/l(k))-1}\  \leq a.
\ee

Putting together \eqref{rewritemes}, \eqref{claim} and \eqref{brac2}, we conclude that for all $K\geq L/(1-L)$ and all $\mathbf{i}  \in D^{\mathbb{N}}$
there exists  $ K\leq k \leq K/ \theta $ such that 
$$\mu\big(Q_k(\mathbf{i})\big) \ \geq\     m^{-dk} a^{kL\log L/\log(L/2\theta)}
 \ \geq\     m^{-dk} a^{2kL\log L/\log(1/\theta)}
 \ =\ m^{-k(d+\epsilon(\theta))},$$
as $\theta \leq L^2/4$, where
$$\epsilon(\theta)\ =\  \frac{ -(\log a)\, 2L\log L}{\log m \log(1/\theta)} \ 
=\  \frac{2 \log(\log_m n)\log a}{\log n} \frac{1}{-\log \theta}.$$

Geometrically this means that for $K\geq L/(1-L)$ every point $z \in F$ belongs to at least one approximate square, $Q_{k(z)}$ say, with $K \leq {k(z)} \leq K/\theta$ and with $\mu(Q_{k(z)}) \geq m^{-k(d+\epsilon(\theta))}$. Since the approximate squares form a nested hierarchy we may choose a subset  $\{Q_{k(z_n)}\}_{n=1}^N \subset \{ Q_{k(z)}: z\in F\}$ that is disjoint (except possibly at boundaries of approximate squares) and which cover $F$. Thus
$$1\ =\ \mu(F)\ = \ \sum_{n=1}^N \mu (Q_{k(z_n)})\ \geq\   \sum_{n=1}^N m^{-k(z_n)(d+\epsilon(\theta))} \ \geq\   \sum_{n=1}^N (2^{-1/2}|Q_{k(z_n)}|)^{(d+\epsilon(\theta))} $$
where  $| Q_k|$ denotes the diameter of the approximate square $Q_k$, noting that $| Q_k|\leq 2^{1/2} m^{-k}$. It follows that $ \overline{\dim}_\theta F \leq d+\epsilon(\theta)$ as claimed.
\end{proof}

The following lemma brings together some basic estimates that we will need to obtain a lower bound for the intermediate dimensions of $F$.

\begin{lem}\label{mcmass}
Let $\epsilon >0$. There exists $K_0\in \mathbb{N}$ and a set $E\subset F$ with  $\mu(E) \geq \frac{1}{2}$ such that for all $\mathbf{i}$ with $\tau(\mathbf{i})\in E$ and $k\geq K_0$,
\begin{equation}\label{uppermass}
\mu(Q_k(\mathbf{i}))\leq m^{-k(d-\epsilon)}
\end{equation}
and
\begin{equation}\label{lowermass}
\mu(R_k(\mathbf{i}))\geq \exp(-k(H(\mu)+\epsilon)),
\end{equation}
where $d=\dim_\textup{H} F$ and $H(\mu) \in (0,\log |D|)$ is the entropy of the measure $\mu$.
\end{lem}
\begin{proof}
McMullen \cite[Lemmas 3,4(a)]{mcmullen} shows that for $\widetilde{\mu}$-almost all $\mathbf{i}  \in D^{\mathbb{N}}$
$$\lim_{k\to\infty}  \mu(Q_k(\mathbf{i}))^{1/k} \to m^{-d}.$$
Thus by Egorov's theorem we may find a set $\widetilde{E}_1\subset D^{\mathbb N}$ with $\widetilde{\mu}(\widetilde{E}_1) \geq \frac{3}{4}$, and 
$K_1\in \mathbb{N}$ such that \eqref{uppermass} holds for all $\mathbf{i}  \in \widetilde{E}_1$ and $k\geq K_1$.

Furthermore, it is immediate from the Shannon-MacMillan-Breimann Theorem and \eqref{rect} that for $\widetilde{\mu}$-almost all $\mathbf{i}  \in D^{\mathbb{N}}$, 
$$\lim_{k\to\infty}  \mu(R_k(\mathbf{i})))^{1/k} \to  \exp(-H(\mu)),$$
and again by Egorov's theorem there is a set $\widetilde{E}_2\subset D^{\mathbb N}$ with $\widetilde{\mu}(\widetilde{E}_2) \geq \frac{3}{4}$, and 
$K_2\in \mathbb{N}$ such that \eqref{lowermass} holds for all $\mathbf{i}  \in \widetilde{E}_2$ and $k\geq K_2$.
The conclusion of the lemma follows taking $E =\tau(\widetilde{E}_1\cap \widetilde{E}_2)$ and $K_0 = \max\{K_1,K_2\}$.
\end{proof}

We now obtain a lower bound for $\underline{\dim}_{\theta}F$, showing in particular that $\underline{\dim}_{\theta}F> \hdd F$ for all  $\theta>0$.

\begin{prop}\label{proplb}
Let $F$ be the Bedford-McMullen carpet as above. Then for $0\leq \theta \leq 1$, 
\begin{equation}\label{lowerbnd}
\underline{\dim}_{\theta}F\geq \hdd F+\theta\frac{\log |D|-H(\mu)}{\log m}.
\end{equation}
\end{prop}

\begin{proof}
Fix $\theta\in(0,1)$, let $E\subset F$ and  $K_0$ be given by Lemma \ref{mcmass}, and let $K\geq K_0$. We define a measure $\nu_K$ which assigns equal mass to all level-$K$ rectangles, and then subdivides this mass among sub-rectangles using the same weights as for the measure $\mu$, given by \eqref{pi}. This gives a measure to which we can apply the mass distribution principle, Lemma \ref{mdp}. Thus for $k\geq K$, writing $b_\ell = a_\ell^{(\log_n m -1)}/m^d$ as in \eqref{pi}, 
\begin{equation}\label{defnuk}
\nu_K\big( R_k (i_1,\ldots,i_k)\big)\ :=\  |D|^{-K}b_{i_{K+1}}\cdots b_{i_k}\ 
=\  |D|^{-K} m^{-(k-K)d} ( a_{i_{K+1}} \cdots a_{i_k})^{L-1}.
\end{equation}

We now consider an approximate square $Q_k(\mathbf{i})$ containing the point $\mathbf{i}$. This approximate square is a union of rectangles $R_k(\mathbf j)$ which, as explained in our comment before \eqref{squaremes1}, each have the same $\mu$ measure equal to $\mu(R_k(\mathbf i))$. The same argument gives that for any $R_k(\mathbf j)\subset Q_k(\mathbf i)$
\[
\nu_K(R_k(\mathbf j))=\nu_K(R_k(\mathbf i))=\mu(R_k(\mathbf i)) \dfrac{|D|^{-K}}{\mu(R_K(\mathbf i))}
\]
where the final equality holds since the formula for $\nu_{K}$ differs from that of $\mu$ only in the mass it assigns according to the first $K$ letters. Putting this together allows one to express the $\nu_{K}$-measure of an approximate square of side length $m^{-k}$  in relation to the $\mu$-measure of such a square. For $\tau(\mathbf{i})\in E$ and $k\geq K$, the approximate square $Q_k(\mathbf{i})$ that contains the point $\mathbf{i}$ has $\nu_K$-measure
\begin{eqnarray}
\nu_{K}(Q_k(\mathbf{i}))&=&\dfrac{|D|^{-K}}{\mu(R_K(\mathbf{i}))}\mu(Q_k(\mathbf{i}))\label{nukmu}\\
&\leq& \dfrac{|D|^{-K} m^{-k(d-\epsilon)}}{\exp(-K(H(\mu)+\epsilon))},\nonumber
\end{eqnarray}
using Lemma \ref{mcmass}. (Alternatively \eqref{nukmu} may be verified directly using \eqref{rect}, \eqref{squaremes} and \eqref{defnuk}.)
Then 
\begin{eqnarray}
\nu_{K}(Q_k(\mathbf{i}))& \leq & m^{-k(d-\epsilon) -K \log_m \big(|D|\exp(-H(\mu)-\epsilon)\big)}\nonumber\\
&\leq& m^{-k\big(d-\epsilon +\frac{K}{k}(\log |D|-H(\mu)-\epsilon)/\log m\big)}\label{inequal}.
\end{eqnarray}
We need bounds that are valid for all $k\in[K,K/\theta]$ corresponding to approximate squares of sides between (approximately) $m^{-K}$ and $m^{-K/\theta}$. The exponent in \eqref{inequal} is maximised when $k  = K/\theta$, so that 
\begin{equation}\label{munu}\nu_{K}(Q_k(\mathbf{i}))\ \leq  \ m^{-k\big(d-\epsilon+\theta(\log |D|-H(\mu)-\epsilon)/\log m\big)}\end{equation}
for all $\mathbf{i}$ with $\tau(\mathbf{i})\in E$ and integers $k\in[K,K/\theta]$, where $\mu (E) \geq \frac{1}{2}$.

To use our mass distribution principle we need equation \eqref{munu} to hold on a set of $\mathbf{i}$ of large $\nu_K$ mass, whereas currently we have that it holds on a set $E$ of large $\mu$ mass. 
Let 
$$E'=\tau\{\mathbf{i} : \mbox{inequality \eqref{munu} is satisfied for all } k\in[K,K/\theta] \}.$$
 Firstly we observe that $Q_k(\mathbf{i})$ depends only on $(i_1,\ldots i_k)$, and we are dealing with $k\leq K/\theta$, so the question of whether $\tau(\mathbf{i})\in E'$ is independent of $(i_{\lfloor K/\theta\rfloor+1}, i_{\lfloor K/\theta\rfloor+2},\ldots)$. Secondly, since $Q_k(\mathbf{i})$ is a union of rectangles $R_k(\mathbf{i})$, and $\nu_K\big( R_k (i_1,i_2,\ldots))$ is independent of $(i_1,\ldots ,i_K)$, the question of whether $\tau(\mathbf{i})\in E'$ is independent of $(i_1,\ldots ,i_K)$. Thus we can write
\[
E'=\bigcup_{i_{K+1}\ldots i_{\lfloor K/\theta\rfloor}\in I''} \bigg(\bigcup_{i_1\ldots i_K\in D^K} R_{\lfloor K/\theta\rfloor}(i_1,\ldots,i_{\lfloor K/\theta\rfloor})\bigg)
\]
for some set $I''\subset D^{\lfloor K/\theta\rfloor-K}$.
But using \eqref{defnuk}  gives
\begin{eqnarray*}
\nu_K \bigg(\bigcup_{i_1\ldots i_K\in D^K} R_{\lfloor K/\theta\rfloor}(i_1,\ldots, i_{\lfloor K/\theta\rfloor})\bigg)&=& \sum_{i_1\ldots i_K\in D^K}\frac{1}{|D|^K}b_{i_{K+1}}\cdots b_{i_{\lfloor K/\theta\rfloor}}\\
&=& b_{i_{K+1}}\cdots b_{i_{\lfloor K/\theta\rfloor}}
\end{eqnarray*}
and  \eqref{rect}  gives
\begin{eqnarray*}
\mu \bigg(\bigcup_{i_1\ldots i_K\in D^K} R_{\lfloor K/\theta\rfloor}(i_1,\ldots, i_{\lfloor K/\theta\rfloor})\bigg)&=& \sum_{i_1\ldots i_K\in D^K}b_{i_1}\cdots b_{i_K}b_{i_{K+1}}\cdots b_{i_{\lfloor K/\theta\rfloor}}\\
&=& \bigg(\sum_{i_1\ldots i_K\in D^K}b_{i_1}\cdots b_{i_K}\bigg)b_{i_{K+1}}\cdots b_{i_{\lfloor K/\theta\rfloor}}\\
&=& b_{i_{K+1}}\cdots b_{i_{\lfloor K/\theta\rfloor}},
\end{eqnarray*}
as $\sum_{i\in D} b_i  =1$. Since these quantities are equal we conclude that
\[
\nu_K(E')\ =\ \mu(E')\ \geq\ \mu(E)\ \geq\ {\textstyle \frac{1}{2}}
\] 
as required.

Since \eqref{munu} holds for all $\mathbf{i}  \in E'$, a straightforward variant on our mass distribution principle, where we use approximate squares instead of balls, gives
\[
\underline{\dim}_{\theta}F\ \geq\  \hdd F-\epsilon +\theta\frac{\log |D|-H(\mu)-\epsilon}{\log m}.
\]
Since $\epsilon$ can be chosen arbitrarily small, \eqref{lowerbnd} follows.
\end{proof}

Note that, in \eqref{lowerbnd}, 
\[
H(\mu)\  =\ -m^{-d}\sum_{\ell\in D }(a_\ell^{L-1}\big((L-1)\log a_\ell - d\log m\big)\ \leq\  \log | D|
\] 
with equality if and only if $\mu$ gives equal mass to all cylinders of the same length, which happens if and only if each column in our construction contains the same number of rectangles. This happens exactly when the box and Hausdorff dimension coincide. Thus our lower bounds give that $\underline{\dim}_{\theta}F> \hdd F$ whenever $\theta>0$, provided that the Hausdorff and box dimensions of $F$ are different.

Since the measures that we have constructed to give lower bounds on $\underline{\dim}_{\theta}F$ are rather crude, it is unlikely that our lower bound for $\underline{\dim}_{\theta}F$ converges to ${\dim}_\textup{B} F$ as $\theta\to 1$. However, a lower bound which \emph{does} approach ${\dim}_\textup{B} F$ as $\theta\to 1$ is given by Proposition \ref{assouad}, noting that $\dim_\textup{A} F > \underline{\dim}_\textup{B} F = {\dim}_\textup{B} F$ provided ${\dim}_\textup{B} F > {\dim}_\textup{H} F$, see \cite{mackay,Fra}.

Many questions on the intermediate dimensions of Bedford-McMullen carpets remain, most notably finding the exact forms of 
$\underline{\dim}_{\theta}F$ and $\overline{\dim}_{\theta}F$. In that direction we would at least conjecture that these intermediate dimensions are equal and strictly monotonic. One might hope to get better estimates using alternative definitions of  $\mu$ in Proposition \ref{propup}  and $\nu_K$ in Proposition \ref{proplb}, but McMullen's measure and our modifications seemed to work best when optimising mass distribution type estimates across $F$ and over the required range of scales.

\section*{Acknowledgements}

JMF was financially supported by a \emph{Leverhulme Trust Research Fellowship} (RF-2016-500) and  KJF and JMF were financially supported in part by an \emph{EPSRC Standard Grant} (EP/R015104/1).  We thank James Robinson for pointing out the reference \cite{KP}.
\\

\bibliographystyle{plain}

{\small
Kenneth J. Falconer, E-mail: \texttt{kjf@st-andrews.ac.uk}
\vskip.2em
Jonathan M. Fraser, E-mail: \texttt{jmf32@st-andrews.ac.uk}
\vskip.2em
Tom Kempton, E-mail: \texttt{thomas.kempton@manchester.ac.uk}
\vskip.2em
}

\bigskip
\end{document}